\documentclass[11pt]{amsart}
\usepackage{amscd,amssymb,graphics}

\usepackage{amsfonts}
\usepackage{amsmath}
\usepackage{amsxtra}
\usepackage{latexsym}
\usepackage[mathcal]{eucal}

\usepackage{graphics,colortbl}

\usepackage[final]{pdfpages}

\input xy
\xyoption{all}
\usepackage{epsfig}
\usepackage[hidelinks]{hyperref} 

\newtheorem{thm}{Theorem}[section]
\newtheorem{cor}[thm]{Corollary}

\newtheorem{prop}[thm]{Proposition}

\theoremstyle{definition}

\theoremstyle{remark}

\numberwithin{equation}{section}

\numberwithin{equation}{section}

\newcommand{\delete}[1]{} 

\newcommand{\ben}{\begin{enumerate}}
	
	\newcommand{\een}{\end{enumerate}}
\newcommand{\bit}{\begin{itemize}}
	
	\newcommand{\eit}{\end{itemize}}
\newcommand{\rest}{\upharpoonright}

\def\ep{{\varepsilon}}
\def\al{\alpha}

\newcommand{\Del}{\Delta}

\newcommand{\card}{\rm{card\,}}

\def\Z {{\mathbb Z}}

\def\T {{\mathbb T}}

\newcommand{\br}{\vspace{4 mm}}

\def\Aut{{\mathrm Aut}\,}

\def\Homeo{{\mathrm{Homeo}}\,}

\def\QED{\nobreak\quad\ifmmode\roman{Q.E.D.}\else{\rm Q.E.D.}\fi}

\begin{document}

\title[]{Short proofs of theorems of Malyutin and Margulis}

\author{Eli Glasner}

\address{Department of Mathematics\\
     Tel Aviv University\\
         Tel Aviv\\
         Israel}
\email{glasner@math.tau.ac.il}


\begin{date}
{April, 2016}
\end{date}

\begin{abstract}
The Ghys-Margulis alternative asserts that a
subgroup $G$ of homeomorphisms of the circle which does not contain a free subgroup on two generators must admit an invariant probability measure. 
Malyutin's theorem classifies minimal actions of $G$.
We present a short proof of Malyutin's theorem and then deduce Margulis' theorem which confirms the G-M alternative.
The basic ideas are borrowed from the original work of Malyutin but the use of the apparatus of the enveloping semigroup enables us to shorten the proof considerably.

\end{abstract}

\maketitle

Our goal in this work is to present a short proof of a neat theorem of Malyutin \cite{Mal}
(Theorem \ref{Mal} below). 
The basic ideas are borrowed from the original work of Malyutin but the use of
the apparatus of the enveloping semigroup enables us to shorten the proof considerably.
As was shown by Malyutin the result known as the Ghys-Margulis alternative 
\cite{Mar} follows, with a little extra work, from his theorem. 
After recalling some background we prove Malyutin's theorem in the first section,
and in the second we provide complete, and simplified, arguments leading to Margulis' theorem.
For more information on the theory of topological dynamical systems and their enveloping 
semigroups see e.g. \cite{E}, \cite{Gl}.

\br

\section*{Some background}
Let $G$ be an infinite countable group. 
An action of $G$ 
on a compact Hausdorff space $X$ is given by a
homomorphism $\rho : G \to \Homeo(X)$. 
We usually omit to mention $\rho$ and write $gx$ instead of
$\rho(g)x$, ($x \in X$) and denote the resulting {\em dynamical
system (or $G$-space)} as $(X,G)$.
Such a system is {\em minimal} if every orbit $Gx$ is dense.
A point $x \in X$ in called an {\em almost periodic} point (or a {\em minimal} point)
if its orbit closure $\overline{Gx}$ is minimal.

A pair $(x, x') \in X \times X$ is called {\em proximal}
if for every neighborhood $V$ in $X \times X$ of the diagonal 
$\Del_X =\{(z, z) : z \in X\}$,
there is $g \in G$ with $(gx, gx') \in V$.
For every system $(X,G)$ the set
$P = \{(x, x') \in X \times X :  x, x' \ \text{are proximal}\}$
is called the {\em proximal relation} on $X$. It is clearly
reflexive, symmetric and $G$-invariant. 
The system $(X,G)$ is {\em proximal} if  $P = X \times X$.
A subset $A \subset X$ is {\em contractible} if there is a point $z \in X$
such that for every neighborhood $U$ of $z$ there is $g \in G$ with
$g(A) \subset U$.
We say that $(X,G)$ is {\em extremely proximal} if every
proper closed subset of $X$ is contractible \cite{Gl}. When $(X,G)$ is minimal 
this is the same as requiring that
for every closed subset $A \subsetneq X$ and every nonempty open 
subset $U \subset X$ there exists an element $g \in G$
with $g(A) \subset U$.

In the sequel we will need the following simple well known result (\cite{Gl}, Proposition 3.5):

\begin{prop}\label{pingpong}
 If a group $T$ admits a non-trivial (i.e. having more than two points) extremely proximal
action $(X,T)$, then $T$ contains a free subgroup on two generators.
\end{prop}

\begin{proof}
Let $U$ and $V$ be disjoint non-empty open subsets in $X$ such that
$U \cup V \not= X$. 
Let $U_1 , U_2$ and $V_1 , V_2$ be disjoint non-empty subsets
in $U$ and $V$ respectively. Then there exists $s$ and $t$ in $T$ such that
$t(X\setminus U_1) \subset U_2$ and $s(X\setminus V_1) \subset V_2$. 
Hence we also have $t^{-1}(X\setminus U_2) \subset U_1$
and $s^{-1}(X\setminus V_2) \subset V_1$. 
Let $S$ be the subgroup of $T$ generated by $s$
and $t$ and let $x \in X \setminus (U \cup V)$. 
If $w$ is a reduced word in $t^{\pm 1}$ and
$s^{\pm1}$ then it is easy to see that $wx \in U_1 \cup U_2 \cup  V_1 \cup V_2$. 
Thus $wx \not= x$ and $w  \not = id$; i.e. $S$ is a free group. 
\end{proof}

Let $E = E(X,G)$ be the enveloping semigroup of the system
$(X,G)$ and fix a minimal left ideal $I \subset E$.
We denote by $J \subset I$ the (nonempty) set of idempotents.
We fix some distinguished element $u \in J$. We let 
$\mathfrak{G} = uI = \{p \in I : up =p\}$; a subgroup of the semigroup $I$.
We also choose a distinguished point $x_0 \in X$,
with $ux_0 = x_0$. The evaluation map
$\eta : (E, G) \to (X, G),\ \eta(p) = px_0$ is a homomorphism
of dynamical systems.  
For each $\al \in \mathfrak{G}$ the map $R_\al : p \mapsto p\al$
defines an automorphism of the minimal dynamical system $(I,G)$
and, in fact, the collection $\{R_\al : \al \in \mathfrak{G}\}$ 
coincides with the group $\Aut(I,G)$ of automorphisms of the system $(I,G)$.
For each $v \in J$ the set $vX = v\mathfrak{G}x_0$ is a {\em maximal almost 
periodic set}, that is, for every finite subset of $n$ distinct points
$\{x_1, x_2,\dots, x_n\} \subset vX$, the point
$(x_1, x_2,\dots, x_n)$ is an almost periodic point of the 
$n$-fold product system $X \times X \cdots \times X$.
In particular, no pair $(x_i, x_j)$ with $i \not= j$ is proximal.
A minimal system $(X,G)$ is called {\em regular} if for any pair of points
$x, x' \in X$ there is an automorphism $\phi \in \Aut(X,G)$
such that $(\phi(x), x') \in P$ \cite{Au}. It is well known and not hard to check
that a minimal system $(X,G)$ is regular iff the map $\eta : I \to X$ is
an isomorphism of dynamical systems.
Note that every minimal proximal system is regular,
and that if $(X,G)$ is regular with the group $\Aut(X,G)$ finite, or even
compact (in the topology of uniform convergence), then
the quotient system $X/ \Aut(X,G)$ is minimal and proximal.

\section{Malyutin's theorem}

We now state and prove Malyutin's theorem:

\begin{thm}\label{Mal}
Every minimal system $(X, G)$, a continuous action of $G$ on the circle $X = \mathbb{T} = \T/\Z$, 
is either equicontinuous or it is a finite to one extension of an extremely proximal action
of $G$ on $\T$, 
where the factor map is a group extension, hence a covering map.
\end{thm}

\begin{proof}
{\bf Step 1:}
We will show first that if the action is distal then it is equicontinuous.
It suffices to show that for sequences $x_i \to x$ and $g_i \in G$
with $ g_i x_i \to y$ and $g_i x \to x$, these assumptions imply that 
$x =y$.
Now the sequence of shrinking arcs whose end points are $\{x_i, x\}$
are mapped by the $g_i$'s onto a sequence of arcs with end points
$\{g_ix_i, g_ix\}$ and, if $x \not=y$, these converge to a
proper interval $K$ with end points $\{x, y\}$. But then, for any closed
proper subarc $L$, with end points $\{a, b\}$,  $a \not=b$
and $\{a,b\} \cap \{x,y\} =\emptyset$, the sequence of arcs
$g_i^{-1}(L)$ shrinks to the singleton $\{x\}$, contradicting
distality. Thus we must have $x =y$ and the equicontinuity is proved. 

\br

%

{\bf Step 2:}  
So we now assume that there are distinct points $a, b$ in $\T$
which are proximal, i.e. for some sequence $g_i \in G$ and a
a point $z \in \T$ we have $g_ia \to z$ and $g_i b \to z$. Thus,
there is a proper closed arc $V$ whose end points are $a$ and $b$
such that $g_i(V) \to \{z\}$. In other words $V$ is contractible.

\br

{\bf Step 3:}  
We use the observation of the previous step to show that if $(X,G)$
is proximal then it is extremely proximal. To see this let, for each $n \ge 3$,
 $\{x_1, x_2, \dots, x_n\}$ be a sequence of $n$ points on $X$ equally spaced
 and ordered in the positive direction. By proximality there is a sequence $g_i \in G$
 and a point $z \in X$ such that $g_i x_k \to z$ for every $1 \le k \le n$
 (use induction).
 It follows that for each consecutive pair $x_k, x_{k+1}$ one of the two closed
 arcs defined by this pair is contractible to $z$ via the sequence $g_i$. However,
 it can not be the case that all the small arcs formed in this way are contractible under $g_i$.
We conclude that at least one large arc is contractible.
 
By compactness we conclude that there is a point $z \in \T$ with the property
that for every open small arc $V$ around $z$ and every closed subset $K \subset \T$
with $z \not\in K$ there is some $g \in G$ with $gK \subset V$.
As the same holds for every point in the $G$-orbit of $z$ the extreme proximaility is proven.

\br

{\bf Step 4:}
We next show that the proximal relation
$P \subset X \times X$ is an open set. 
Let $(x, x') \in P$.
By Step 2 there is a contractible open arc $\emptyset \not= V \subset X$.
By minimality, fixing any point $w \in V$, for some sequence 
$h_n \in G$, $h_n x \to w$ and $h_n x' \to w$. 
Choosing  a sufficiently large
$n$ and sufficiently small open arcs $U$ and $U'$ around $x$ and $x'$
respectively we will have $h_n(U), h_n(U') \subset V$.
Finally, as $V$ can be shrunk to a point by a sequence $g_i$, we conclude
that $U \times U' \subset P$.

\br

{\bf Step 5:}
We claim that, for each $v \in J$, the maximal 
almost periodic set $vX$ is finite. In fact,
assuming the contrary, there is an infinite sequence
$\{x_n \} \subset vX$ and we can assume that it converges
to some point $x \in X$. By minimality
there is $g \in G$ with $gx \in V$, where $V$ is a proper
contractible arc, 
and then for all sufficiently large $n$
we have $gx_n \in V$. This however is impossible as every pair of
distinct points in $vX$ is distal.
It now follows that there is $k \ge 1$ such that for all
$v \in J$, $vX$ has exactly $k$ points (recall that
for every $v, w \in J$ we have $vw=v$, hence $vwX = vX$
and $wvX = wX$).

Moreover, as $uX = \mathfrak{G}x_0$, it follows that,
with $\mathfrak{H}=\mathfrak{G}(X,x_0) = \{\al \in \mathfrak{G} : \al x_0 = x_0\}$ 
(the Ellis group at $x_0$),  we have 
$\card \mathfrak{G}/\mathfrak{H} = k$. 
Since for every $\al \in \mathfrak{G}$, 
$ \mathfrak{G}(X,\al x_0) = \al \mathfrak{G}(X, x_0)\al^{-1}$,
we conclude that 
$\bigcap_{\al \in \mathfrak{G}}\al \mathfrak{H} \al^{-1}= \{id\}$, 
whence that $\mathfrak{G}$ is a finite group. 
Thus, the regular dynamical system $(I,G)$ is
a finite group extension of a proximal system. 

\br

{\bf Step 6:}
Note that if $\phi \in \Aut(X,G)$ is an automorphism of the
system $(X,G)$, then the pair $(x_0,\phi(x_0))$ is an almost
periodic point of the product system $(X \times X, G)$ and therefore there 
is an $\al \in \mathfrak{G}$ with $\phi(x_0) = \al x_0$.
Then, for the evaluation map $\eta : E(X,G) \to X$,  
$\eta(p) = px_0$, we have for every $p \in E$:
$$
(\phi \circ \eta)(p) =
\phi(\eta(p)) = \phi(px_0) = p\phi(x_0) = p\al x_0 = 
\eta(p\al) = (\eta \circ R_\al)(p). 
$$
Thus the set 
$$
\mathfrak{A} := \{\al \in \mathfrak{G} : \eta \circ R_\al 
= \phi \circ \eta \ {\text{for some}}\ \phi 
 \in \Aut(X,G)\}
 $$ 
 is a subgroup of $\mathfrak{G}$
 which is ``mapped" by $\eta$ onto $\Aut(X,G)$.
 It follows that the group $\Aut(X,G)$ is finite
 and, passing to the dynamical system $\tilde{X} =X/\Aut(X,G)
 \cong I / \mathfrak{A}$, we note that the quotient map
 $X \to \tilde{X}$ is a finite group extension, hence a covering map
with $\tilde{X}$ homeomorphic to $\T$.
We repeat this procedure with $\Aut(\tilde{X},G)$
 to obtain $\tilde{\tilde{X}} = \tilde{X}/ \Aut(\tilde{X},G)$ and so on.
As $\mathfrak{G}$ is finite we will eventually end up with a system
whose automorphism group is trivial.
%
%
%
 In order to keep the notation simple we now relabel
 this new system as $X$.
 Thus the results obtained so far are still valid for the
 new system and we now have one more assumption,
 namely that $\Aut(X,G)$ is trivial.

\br

{\bf Step 7:} 
We will finish the proof by showing that unless
$\eta : I \to X$ is one-to-one, $\Aut(X,G)$ necessarily contains
a non identity automorphism.

As we have seen (in Step 3) for each $x \in X$ the proximal cell 
$P[x] = \{x' : (x, x') \in P\}$ is an open arc containing $x$. We define
$\tau(x)$ to be the end point of this arc which goes in the positive direction.
Note that if for some $x \in X$ we have $\tau(x) =x$ then the system $(X,G)$
is proximal, hence regular, and we are done.
Otherwise $\tau(x) \not= x$ for every $x \in X$ and the pair $x, \tau(x)$ is {\bf not} proximal.
Next observe that it follows directly from the definition of $\tau$ that
$\tau(wgx) = g\tau(x)$ for every $x \in X$. This in turn implies that the image of 
$\tau$ is dense in $X$.
Also, as $x$ and $\tau(x)$ are not proximal, it is easy to check that
$\tau$ preserves the circular order. These two facts imply that $\tau$ is continuous, hence also surjective.
Finally we show that $\tau$ is one-to-one. 
Suppose to the conterary that $\tau(x_1) =\tau(x_2)=z$
with $x_1 \not=x_2$.  It immediately follows that there is an open arc
$L$ which is mapped onto $z$. By minimality there is a finite
set $\{g_i\}_{i=1}^n \subset G$ with $X = \cup_{i=1}^n g_i L$, whence
$\tau{X} =  \{g_iz\}_{i=1}^n$. This however is impossible as 
we have seen that $\tau(X)=X$.
Thus $\tau$ is an element of $\Aut(X,G)$ which is not the identity. 

Putting together our results so far, we see that, if the system $(X,G)$ is not equicontinuous,
then it is a $k$ to one extension, for some positive integer $k$,
of an extremely proximal system.
As the fundamental group of $\T$ is $\Z$ it follows that this
is a $\Z_k$ extension and our proof is complete. 
\end{proof}

\section{The Ghys-Margulis alternative} 
Following Malyutin we now retrieve a version of the Ghys-Margulis alternative
\cite{Gh}, \cite{Mar}.
 
\begin{thm}\label{GM}
Let $G$ be a countable subgroup of homeomorphisms of the circle $\T$ which act minimally
on $\T$. Then either $G$ is isomorphic to a subgroup of $\T$ and the action is via
translations, or $G$ contains a free group on two generators.
\end{thm}
 
\begin{proof}
According to Theorem \ref{Mal} the resulting dynamical system $(\T,G)$ is either equicontinuous 
and acts on $\T$ by translations, or it admits a nontrivial factor, via a covering map,
which is a minimal extremely proximal $G$-action on $\T$.
In the latter case we apply Proposition \ref{pingpong} to conclude that $G$ contains
a free group on two generators.
\end{proof}

It is easy to see that an extremely proximal system can not admit an invariant probability measure.
This observation, together with Theorem \ref{GM}, yield the following corollary.

\begin{cor}
A minimal system $(\T,G)$ either admits an invariant probability measure 
or it is a finite to one extension of an extremely proximal action
of $G$ on $\T$, 
where the factor map is a group extension, hence a covering map.
\end{cor}

\br
   
When the minimality assumption is dropped we have the following classification of 
$G$-actions on $\T$ (see \cite[Theorem 3.7]{Bek}).

\begin{thm}
Let $G \subset \Homeo(\T)$ be a group. 
Then there is a nonempty $G$-minimal subset of $\T$
and at least one of the following mutually exclusive assertions holds:
\begin{enumerate}
\item
Every minimal set is finite.
\item
There is a unique minimal set $M \subset \T$, which is perfect and nowhere dense.
\item 
The minimal set coincides with $\T$.
\end{enumerate}
\end{thm}

\begin{proof}
Suppose that $M \subset \T$ is a minimal Cantor set. Pick any maximal open arc $A$ is
the complement of $M$. Then for any $\ep >0$  the countable collection $\{g(A) : g \in G\}$ 
contains only finitely many arcs of length $\ge \ep$. Thus $A$ is a contractible set.
This proves the uniqueness of $M$. The rest of the claims of the theorem follow easily.
\end{proof}

In the second case it is easy to see that, by collapsing the closure
of each maximal open arc in the complement of the minimal set $M$
to a single point, one obtains a factor map $\pi : (\T,G) \to (Y,G)$, where
(i) $Y$ is homeomorphic to $\T$, 
(ii) $\pi \rest M : M \to Y$ is at most two-to-one,
with $\card (\pi\rest M)^{-1}(y) = 2$ iff $\pi^{-1}(y)$ is a closed arc as above. 
Thus, in particular, the set $\{y \in Y :  \card (\pi\rest M)^{-1}(y) = 2\}$ is countable. 

\br

Using these facts we deduce the following theorem, \cite[Corollary 2]{Mal}.

\begin{thm}
For any $G$-dynamical system $(\T,G)$ exactly one of the following alternatives holds.
\begin{enumerate}
\item
There is a finite orbit.
\item
There is a factor map $\pi : (\T,G) \to (Y,G)$ onto a system $(Y,G)$
where $Y \cong \T$ and the $G$-action on $Y$ is minimal and equicontinuous.
\item
There is a factor map $\pi : (\T,G) \to (Y,G)$ onto a system $(Y,G)$
where $Y \cong \T$ and the $G$-action on $Y$ is minimal and extremely proximal.
\end{enumerate}
\end{thm}

In turn, this leads to the following, more usual, formulation of the Ghys-Margulis alternative
\cite{Mar}.

\begin{thm}
A subgroup of $\Homeo(\T)$ which does not contain a free subgroup on two generators
must admit an invariant probability measure. 
\end{thm}

%


\end{document}